\documentclass[reqno]{amsart}
\usepackage{amsfonts}
\usepackage{eurosym}
\usepackage{hyperref}
\usepackage{cases}
\usepackage{amsmath,amscd,commath}

\AtBeginDocument{\vspace{8mm}}

\begin{document}
\title[\hfil criterion of compactness]
{A criterion of compactness in the space of fuzzy numbers and applications}
\author[T. M. Thuyet, D. Huy. Hoang, P. T. Son, H. Q. Duc]
{Tran Minh Thuyet, Do Huy Hoang, Pham Thanh Son, Ho Quang Duc}
\address{Tran Minh Thuyet, Do Huy Hoang, Pham Thanh Son\\
University of Economics of Ho Chi Minh City, 59C, Nguyen Dinh Chieu Str, District 3, Ho Chi Minh City, Vietnam}
\email{tmthuyet@ueh.edu.vn, thanhsonpham27@gmail.com}

\address{Ho Quang Duc \\
High School of Vinh Kim, Chau Thanh, Tien Giang Province}
\email{hoquangductg@gmail.com}

\begin{abstract}
We propose a simple criterion of compactness in the space of fuzzy number on
the space of finite dimension and apply to deal with a class of fuzzy
intergral equations in the best condition.
\end{abstract}

\maketitle

\numberwithin{equation}{section} \newtheorem{theorem}{Theorem}[section] %
\newtheorem{lemma}[theorem]{Lemma} \newtheorem{remark}[theorem]{Remark} %
\newtheorem{definition}[theorem]{Definition} %
\newtheorem{proposition}[theorem]{Proposition} %
\newtheorem{example}[theorem]{Example} \allowdisplaybreaks

\section{\textbf{Introduction}}

In recent decades, we have seen that the theory about fuzzy logic and fuzzy
mathematics have strongly developed and more and more widely penetrated
into many fields of application science such as: decision making, fuzzy
control, neural networks, data analysis, risk assessment, optimization and
transportation,...

A direction contributing to the development of the fuzzy mathematics is the
research on the existence, uniqueness and properties of solution of fuzzy
differential, integral equations, see for instance \cite{HM}. There, the
core tools are fixed point theorems such as contraction principal, Schauder
theorem see for example Ravi P Agarwal et al. \cite{RS}, A Khastan et al. \cite{AJ}
etc.

In \cite{RS}, Ravi P Agarwal et al presented the result for the existence
of solution of fuzzy integral equation 
\begin{equation}\label{eq1.1}
y(t) = {y_0}(t) + \frac{1}{{\Gamma (q)}}\int\limits_0^t {{{(t - s)}^{q - 1}}%
f(s,y(s))ds},
\end{equation}
by using Schauder fixed point theorem in the semilinear metric space. The
main result is stated in Theorem 4.2 with hypothesis 
\begin{equation*}
f:[0;{a^*}] \times A \to \mathbb{E}_c^n\,\, \text{is continuous and compact,
where}\,\, A \subset \mathbb{E}_c^n\,\, \text{is a bounded set}.
\end{equation*}
The authors proved that the operator $T:\Omega \to \Omega \subset C([0,a],%
\mathbb{E}_c^n)$ defined by 
\begin{equation*}
Ty(t) = {y_0}(t) + \frac{1}{{\Gamma (q)}}\int\limits_0^t {{{(t - s)}^{q - 1}}%
f(s,y(s))ds},
\end{equation*}
is continuous and compact.

However, according to our group's opinion, this paper contains three
problems to be necessarily discussed.

\begin{enumerate}
\item In the step of proving continuity of operator $T$ is continuous, we think that: It is necessary to add the hypothesis of uniform
continuity of the function $f$.

\item Besides, in the step of proving the level-equicontinuity of $T(\Omega
)(t)$. Authors used the level-equicontinuity of $y_0(t)$. While
this property is not deduced from the hypothesis which is only given as that ${%
y_0}([0,{a^*}])$ is compact-supported.

\item In the proof of the compact-supported property of $T(\Omega )(t)$ (the
end of page 10). Authors argued that: by using Theorem 2.2 and the relative
compact property of $f([0,{a^*}] \times \Omega )$ we deduce the
compact-supported property of $f([0,{a^*}] \times \Omega )$ . This argument
is not true with Theorem 2.2 only stating that: If $A$ is a compact-supported subset of $%
\mathbb{E}_c^n$, then the two following statements are equivalent.
\end{enumerate}
\begin{itemize}
\item $A$ is a relative compact subset of $\mathbb{E}_c^n,$

\item $A$ is level-equicontinuous on [0,1].
\end{itemize}

Theorem 2.2 did not assert that: The compact-supported property is a consequence of the
relative compact property.

Maybe authors A. Khanstan et al have cared for the shortcomings. So
in \cite{AJ}, they investigated the existence of solution of fuzzy integral
equation: 
\begin{equation}\label{eq1.2}
u(t) = \frac{1}{{\Gamma (q)}}\int\limits_0^t {{{(t - s)}^{q - 1}}f(s,u(s))ds}%
,
\end{equation}
which similar and simpler than \eqref{eq1.1}.

Besides, Theorem 2.2 of paper \cite{RS} is recalled in the label "Theorem 2.3"
 and right below they had a remark that: in fact, the following two
statements

\begin{itemize}
\item $A$ is relative compact subset in $\mathbb{R}_F^c$,

\item $A$ is compact-supported in $\mathbb{R}_F^c$ and level-equicontinuous
on [0,1],
\end{itemize}

are equivalent.

This remark could deal with the third shortcomings in paper \cite{RS}.
However, they did not say that it was cited from which resources?

Next, the authors in \cite{AJ} also claimed that: the operator $\mathcal{T}$
defined by 
\begin{equation*}
\mathcal{T}u(t) = \frac{1}{{\Gamma (q)}}\int\limits_0^t {{{(t - s)}^{q - 1}}f(s,u(s))ds}
\end{equation*}
is continuous and compact. 

As for the detail, they used a different technique from one in 
\cite{RS}. The operator $\mathcal{T}$ is decomposed into two operators $\mathcal{A}$ and $\mathcal{N}$ i.e, 
$\mathcal{T}=\mathcal{A}\circ \mathcal{N}$, where 
\begin{equation*}
{\mathcal{A}}v(t)=\frac{1}{{\Gamma (q)}}\mathop\int \limits_{0}^{t}{%
(t-s)^{q-1}}{s^{-r}}v(s){\mkern1mu}ds,
\end{equation*}%
and 
\begin{equation*}
\mathcal{N}u(t)={f_{r}}(t,u(t))\,\,\,\text{with}\,\,\,{f_{r}}(t,x)={t^{r}}%
f(t,x)\,\,\,\text{for some}\,\,\,r\in (0,q).
\end{equation*}

To prove the continuity and compactness of $\ \mathcal{T}$, they made the
following assumptions (presented in Theorem 3.13)

\begin{itemize}
\item[a)] $f:(0,1] \times \mathbb{R}_F^c \to \mathbb{R}_F^c$ is continuous,

\item[b)] ${f_r}:[0,1] \times \mathbb{R}_F^c \to \mathbb{R}_F^c$ is compact
and uniformly continuous.
\end{itemize}

The strong point of paper \cite{AJ} is the fact that they overcame all
shortcomings of paper \cite{RS}. Further, the result of the existence of
solution is still true in the case that the function $f$ admits a singular at $%
t=0$.

However to overcome the first shortcoming of \cite{RS}, the authors in \cite%
{AJ} must use the hypothesis of uniform continuity of the function $f_r$.

Another observation is the fact that both paper \cite{RS} and \cite{AJ} applied the
result of Theorem 2.2 (stated in \cite{RS}). This theorem relates to a quite
complex concept that is the compact - supported concept.

From the above motivation, in the present paper we wish to propose one simple
approach without using compact - supported concept as well as uniform
continuity of the function $f$ but we still obtain the result of the
existence of solution for equation (1.1).

The layout of this paper is organiged as follows: In Section 2, some
necessary results and concepts are recalled. Specially, we use Section 3 to
present a convenient criterion of compactness in the space of fuzzy number $%
\mathbb{R}_c^F$. Section 4 is used to prove an existence result by using the
above criterion. Finally, in Section 5, we extend the existence result for
the space of fuzzy numbers $\mathbb{E}_c^n$.

\section{\textbf{Preliminaries}}

In this section, we give some definitions and introduce the necessary
notations which will be used throughout this paper.

Let us denote $K_c(\mathbb{R}^n)$ as the family of all nonempty, compact and
convex subsets of $\mathbb{R}^n$. In $K_c(\mathbb{R}^n)$ we define

\begin{itemize}
\item[i)] $A + B = \{ a + b:a \in A,\,\,b \in B\},$

\item[ii)] $\lambda A = \{\lambda a : a\in A\},$ 
\end{itemize}
for all $A, B \in K_c(\mathbb{R}^n),\,\, \lambda \in \mathbb{R}$.

The distance between $A$ and $B$ is defined by the Hausdorff-Pompeiu metric 
\begin{equation*}
{d_H}(A,B) = \max \left\{ {\mathop {\sup }\limits_{x \in A} \mathop {\inf }%
\limits_{y \in B} \parallel x - y\parallel ,\mathop {\sup }\limits_{y \in B} %
\mathop {\inf }\limits_{x \in A} \parallel x - y\parallel } \right\}.
\end{equation*}
$K_c(\mathbb{R}^n)$ is a complete and separable metric space with respect to
the Hausdorff - Pompeiu metric (see \cite{LG}).

In the following, we give some basic notions and results on fuzzy set
theory. We denote by $\mathbb{E}^n$ the space of all fuzzy numbers in $\mathbb{R
}^n$, that is, $\mathbb{E}^n$ is the space of all functions $u:\mathbb{R}^n\to [0,1]$, 
satisfying the following properties (see for example \cite{BB})

\begin{itemize}
\item[(i)] $u$ is normal, i.e. $\exists t_{0} \in\mathbb{R}$ for which $u(t_{0})=1$,
\item[(ii)] $u$ is fuzzy convex, i.e. 
\begin{equation*}
u \bigl(\lambda t_1 + (1-\lambda) t_2 \bigr) \geq\min \{ u({t_1}),u({t_2})\}\quad \text{for any } t_1, t_2\in\mathbb{R}^n, \text{ and } \lambda\in[0,1],
\end{equation*}
\item[(iii)] $u$ is upper semi-continuous,
\item[(iv)] $cl(\text{supp}u)$ is compact.
\end{itemize}

The fuzzy null set is defined by 
\begin{equation*}
0(t) = \left\{ 
\begin{array}{l}
0,\quad t \ne 0, \\ 
1,\quad t = 0.%
\end{array}
\right.
\end{equation*}

If $u\in \mathbb{E}^n$, then the set 
\begin{equation*}
[u]^\alpha = \left\{ {t \in \mathbb{R}^n | u(t) \ge \alpha }
\right\},\;\;\,\alpha \in (0,1],
\end{equation*}
\begin{equation*}
{[u]^0} = \overline {\left\{ {t \in \mathbb{R}^n | u(t) > 0} \right\}} =
cl\left\{ {t \in \mathbb{R}^n | u(t) > 0} \right\}
\end{equation*}
is called the $\alpha$ - level set of $u$. From the definition of $\mathbb{E}%
^n$, we can prove that 
\begin{equation*}
\forall u\in \mathbb{E}^n, [u]^\alpha \in K_c(\mathbb{R}^n),\,\,\, \text{for
every}\,\,\, \alpha \in [0,1].
\end{equation*}
We usually denote $\mathbb{E}^1$ by $\mathbb{R}_F$ and if $u\in \mathbb{E}^1$%
, we denote 
\begin{equation*}
{[u]^\alpha } = [u_\alpha ^ - ,u_\alpha ^ + ],\;\;\,\text{for every}%
\,\,\,\alpha \in [0,1].
\end{equation*}

According to Zadeh's extension principle, we have the addition and the
scalar multiplication in fuzzy-number space $\mathbb{E}^n$ as usual. It is
well known that 
\begin{equation*}
[u+v ]^{\alpha} = [u ]^{\alpha} + [v ]^{\alpha} , \qquad [k u ]^{\alpha} = k
[u ]^{\alpha},\,\,\,\forall \alpha\in [0,1], \,\,\, k\in \mathbb{R}.
\end{equation*}

The metric in $\mathbb{E}^n$ is defined by 
\begin{equation*}
D(u,v) = \sup_{\alpha\in [0,1]} d_H \bigl( \lbrack u]^{\alpha} , [v]^{\alpha
} \bigr),
\end{equation*}
we have the following properties (see \cite{LM}):

\begin{itemize}
\item[i)] $(\mathbb{E}^n,D )$ is a complete metric space,

\item[ii)] $D(u+w,v+w) = D(u,v)$,

\item[iii)] $D(\lambda u,\lambda v) = \vert \lambda\vert D(u,v)$,

\item[iv)] $D(\lambda u,\mu u) = |\lambda - \mu |D(u, 0)$,

\item[v)] $D(u+w,v+t) \leq D(u,v)+ D(w,t)$,
\end{itemize}

for all $u,v,w,t\in \mathbb{E}^{n}$ and $\lambda ,\mu \in \mathbb{R}$.\\

Let $T=[a,b]\subset \mathbb{R}$ be a compact interval.

\begin{definition}\rm
    A mapping \(F:T\rightarrow \mathbb{R}_F\) is \textit{strongly measurable} if, for all \(\alpha \in[0,1]\) the set-valued function 
\(F_{\alpha}:T\rightarrow K_c(\mathbb{R}^n)\) defined by the following:
$$F_{\alpha}(t)=[F(t)]^{\alpha},\,\,\,\,t\in T,$$ 
is Lebesgue measurable.\\

A mapping \(F:T\rightarrow \mathbb{R}_F\) is called \textit{integrably bounded} if there exists an integrable function $k:T\to \mathbb{R}_+$ such that
 \(D(F_0(t),0)\leq k(t)\) for all \(t\in T\).
\end{definition}%
\begin{definition}\rm
    Let \(F:T\rightarrow \mathbb{R}_F\). The integral of $F$ over $T$, denoted by \(\int_{T}F(t)\,dt\), is defined by 
the following expression
\begin{equation*}
\begin{split}
\biggl[\int_{T}F(t)\,dt \biggr]^{\alpha}&{} = \int _{T}F_{\alpha}(t)\,dt\\
&{}= \biggl\{ \int_{T}f(t)\,dt  | f:T\rightarrow\mathbb{R} \text{ is a measurable selection for } 
F_{\alpha} \biggr\},\quad \alpha\in\,[0,1].
\end{split}
\end{equation*}
\end{definition}
A strongly measurable and integrably bounded mapping $%
F:T\rightarrow E$ is said to be \textit{integrable} over $T$ if $\int_{T}F(t)\,dt\in \mathbb{R}_F$.  

\begin{proposition}
If \(F:T\rightarrow \mathbb{R}_F\) is continuous, then $F$ is integrable.
\end{proposition}
\begin{theorem}(see \cite{KO})
Let $F, G:T\to \mathbb{R}_F$ be integrable and $\lambda\in \mathbb{R}$. Then
\begin{enumerate}
\item [i)] $\int\limits_T {(F + G)}  = \int\limits_T F  + \int\limits_T G $,
\item [ii)] $\int \limits_T \lambda F = \lambda \int \limits_T F$,
\item [iii)] $D(F,G)$ is integrable on $I$,
\item [iv)] $D(\int \limits_T F,\int \limits_T G) \le \int \limits_T D(F,G)$.
\end{enumerate}
\end{theorem}
We denote $\mathbb{E}_{c}^{n}$ the space of $u\in \mathbb{E}^{n} $ with the property that the mapping 
$[0,1]\rightarrow K_{c}(\mathbb{R}^{n}),\,\,\alpha \mapsto {[u]^{\alpha }}$ is continuous. It is known that 
$(\mathbb{E}_{c}^{n},D)$ is a complete metric space, $\mathbb{E}_{c}^{1}$ is usually denoted by 
$\mathbb{R}_{F}^{c}$.

The concept of semilinear space and related concepts were already considered,
for instance in \cite{RS}. A semilinear metric space is a semilinear space 
$S$ with a metric $d:S \times S \to R$ which is translation invariant and
positively homogeneous, that is,

\begin{itemize}
\item[i)] $d(a+c,b+c)=d(a,b)$,

\item[ii)] $d(\lambda a,\lambda b) = \lambda d(a,b)$, for all $a, b\in S,
\lambda \ge 0.$
\end{itemize}

If $S$ is a semilinear metric space, then addition and scalar multiplication
on $S$ are continuous. Furthermore, if $S$ is complete, then we say that $S$ is a
semilinear Banach space. We say that a semilinear space $S$ has the
cancellation property if $a+c=b+c$ implies $a=b$ for every $a,b,c \in S$.

It is known that $(\mathbb{E}^n,D ),\,\,(\mathbb{E}_c^n,D ),\,\,C([a,b],%
\mathbb{E}^n),\,\,C([a,b],\mathbb{E}_c^n)$ are semilinear Banach spaces
having the cancellation property.

By using the fact that a semilinear metric space $S$ having cancellation
property can be isometrically embedded into a normed space we can prove
Schauder's fixed point theorem for $S$. 
\begin{theorem}(see \cite{RS})
Let $B$ be a nonempty, closed, bounded and convex subset of a semilinear Banach space $S$ having the cancellation property,
 and suppose that $P:B \to B$ is a compact operator. Then $P$ has at least one fixed point in $B$.
\end{theorem}
\begin{remark}\rm
Here, it is understood that: a compact operator is a continuous operator mapping bounded sets into relative compact sets.
\end{remark}

\section{\textbf{A criterion of compactness in $\mathbb{R}_F^c$}}

In this section we propose a convenient criterion of compactness for $%
\mathbb{R}_F^c$ to deal with a problem in Section 4.

First, we recall two conventional forms of Ascoli - Arzela Theorem. 
\begin{theorem} Let $E$ be a complete metric space and $K$ be a compact metric space. Then a subset $A$ of $C(K,E)$
 is relatively compact \textbf{iff}
\begin{enumerate}
\item [i)] for every $r\in K,\,\,\, A(r)=\{u(r)\,\,|\,\,\,u\in A\}$ is relatively compact in $E$,
\item [ii)] $A$ is equicontinuous.
\end{enumerate}
\end{theorem}
In particular for $E=\mathbb{R}^m$ with $m\in \mathbb{N}$, we have the
following version of Ascoli - Arzela theorem 
\begin{theorem} A subset $A$ of $C(K,\mathbb{R}^m)$ is relatively compact \textbf{iff}
\begin{enumerate}
\item [i)] $A$ is uniformly bounded,\\
(i.e. there is $M>0$ such that $|u(r)|\le M$ for every $(r,u)\in [0,1]\times A$)
\item [ii)] $A$ is equicontinuous.
\end{enumerate}
\end{theorem}
Now we propose a criterion of compactness in $\mathbb{R}_F^c$ as follows 
\begin{theorem} A subset $A$ of $\mathbb{R}_F^c$ is relatively compact \textbf{iff}
\begin{enumerate}
\item [i)] $A$ is bounded in $\mathbb{R}_F^c$,\\
(i.e. there exists a constant $M>0$ such that $D(u,\ 0)\le M$ for all $u\in A$)
\item [ii)] $A$ is level - equicontinuous.\\
(i.e. for every $\varepsilon >0$ there exists $\delta  > 0$ 
such that for $\alpha ,\beta  \in [0,1]$ if $|\alpha  - \beta | < \delta $ then
$\mathop {\sup }\limits_{u \in A} {d_H}({[u]^\alpha },{[u]^\beta }) < \varepsilon$)
\end{enumerate}
\end{theorem}

\begin{proof}
Recall that $X \equiv C([0,1],\mathbb{R}^2)$ is a Banach space with respect
to the sup norm by 
\begin{equation*}
\begin{split}
||f|{|_X} &{}= \mathop {\sup }\limits_{\alpha  \in [0,1]} |f(\alpha )|_{{\mathbb{R}^2}} 
				   = \mathop {\sup }\limits_{\alpha  \in [0,1]}\left|({f_1}(\alpha ),{f_2}(\alpha )\right)|_{{\mathbb{R}^2}}\\
 &{} =\mathop {\sup }\limits_{\alpha  \in [0,1]} \max \left\{|{f_1}(\alpha )|, |{f_2}(\alpha )| \right\}\,\,\,\text{for every}\,\,\,f=(f_1,f_2)\in C([0,1],\mathbb{R}^2),
\end{split}
\end{equation*}
and the induced metric from $||.||_X$ is 
\begin{equation*}
{d_X}(f,g) = ||f - g||_X \,\,\,\text{for every}\,\,\,f,g\in X.
\end{equation*}

Next consider the following mapping $j$ 
\begin{equation*}
\fullfunction{j}{\mathbb{R}_F^c} {C([0,1],\mathbb{R}^2)}{u}{j(u)=(u^-,u^+)}
\end{equation*}
by 
$$ j(u)(\alpha)=\left(u^-(\alpha),u^+(\alpha)\right),\,\,\,\forall \alpha \in [0,1],$$
where $u^ \pm (\alpha)$ is the left and right end - points of $\alpha - $
level set $[u]^ \alpha$ i.e $[u]^\alpha=\left[u^-(\alpha),u^+(\alpha)\right]$.\newline
It is easy to check that 
\begin{equation*}
j(u+v)=j(u)+j(v),\,\,\, j(\lambda u)=\lambda j(u),\,\,\,\forall u,v \in 
\mathbb{R}_F^c\,\,\,\text{and}\,\,\,\lambda \ge 0.
\end{equation*}

Further, for every $u,v\in \mathbb{R}_F^c$, we have 
\begin{equation}
\begin{split}
{d_X}(j(u),j(v)) &{}={\left\| {j(u) - j(v)} \right\|_X}={\left\| {({u^ - } - {v^ - },{u^ + } - {v^ + })} \right\|_X}\\
&{} = \mathop {\sup }\limits_{\alpha \in [0,1]} \max \left\{ {|{u^ - }%
(\alpha ) - {v^ - }(\alpha )|,|{u^ + }(\alpha ) - {v^ + }(\alpha )|} \right\}\\
&{}=\mathop {\sup }\limits_{\alpha  \in [0,1]} {d_H}\left( {{{[u]}^\alpha },{{[v]}^\alpha }} \right)= D(u,v).
\end{split}%
\end{equation}
So the mapping $j$ isometrically embeds $\mathbb{R}_F^c$ into $X$.

From the above isometric embedding, it is easy to check that: $A$
is relatively compact in $\mathbb{R}_F^c$ \textbf{iff} $j(A)$ is relatively compact in $X.$

Furthermore, it is not difficult to check that:

\begin{itemize}
\item The equicontinuity of $j(A)$ is just be level - equicontinuity of $A,$

\item The uniform boundedness of $j(A)$ is just be the boundedness of $A$ in $\mathbb{R}_F^c.$
\end{itemize}

Finally, we end the proof by applying Theorem 3.2 for $j(A)$ in $X$.
\end{proof}

\section{\textbf{Fuzzy integral equation}}

We consider the nonlinear fuzzy integral equation of the form 
\begin{equation}
u(t) = {u_0}(t) + \frac{1}{{\Gamma (q)}}\int\limits_0^t {{{(t - s)}^{q - 1}}%
f(s,u(s))\,ds,\,\,\,t\in[0,a]},  \label{eq3.1}
\end{equation}
where $a>0$ and $0<q<1$ are constants given.\newline

We establish the following assumptions:\newline
$(H_1)\,\,\,\,\,u_0 \in C([0,a],{\mathbb{R}_F^c})$,\newline
$(H_2)\,\,\,\,\,f:(0,a] \times \mathbb{R} _F^c \to \mathbb{R}_F^c$ is a
given continuous and compact mapping.\newline

\begin{remark}\rm
It is possible that $f$ has a singular at $t=0.$
\end{remark}
We put
\begin{itemize}
\item $N = {\sup _{t \in [0,a]}}D({u_0}(t), 0).$
\end{itemize}

Then $N$ is well - defined since $u_0([0,a])$ is compact.\newline
Let $R>N$. We put

\begin{itemize}
\item $G = \{ (t,x) \in (0,a] \times \mathbb{R} _F^c:D(x,0) \le R\},$

\item $M = {\sup _{(t,x) \in G}}D(f(t,x),0).$
\end{itemize}

$M$ is well - defined because of compactness of $f.$

Then, we have the following theorem. 
\begin{theorem}
Assume that $({H_1}),({H_2})$ hold, then the fuzzy integral equation (4.1) has at least one solution $u \in C([0,\eta],{\mathbb{R}_F^c})$ where $\eta  = \min \left\{ {a,\,{{\left[ {\frac{{(R - N)\Gamma (q + 1)}}{M}} \right]}^{\frac{1}{q}}}} \right\}$. 
\end{theorem}

\begin{proof}
For $\eta \in (0,a]$, we put 
\begin{equation*}
X:=C([0,\eta],\mathbb{R}_F^c).
\end{equation*}
Then it is not difficult to check that $X$ is a semilinear complete metric
space with respect to the metric 
\begin{equation*}
{D_X}(u,v) = \mathop {\sup }\limits_{t \in [0,\eta ]} D(u(t),v(t)) \,\,\, 
\text{for every}\,\,\, u,v\in X.
\end{equation*}
We define the set 
\begin{equation*}
\Omega = \left\{ {u \in X\mid D_X(u, 0) \le R} \right\},
\end{equation*}
and 
\begin{equation*}
B_R= \{ x \in \mathbb{R} _F^c:D(x, 0) \le R\}.
\end{equation*}
It is easy to see that $\Omega$ is a closed, bounded and convex subset of
the semilinear Banach space $X$. On the set $\Omega$, we define the operator 
\begin{equation*}
\fullfunction{\mathcal{T}}{\Omega} {X}{u}{\mathcal{T}u}
\end{equation*}
by 
\begin{equation*}
\mathcal{T}u = {u_0} + \frac{1}{{\Gamma (q)}}\mathcal{A}u,
\end{equation*}
where 
\begin{equation*}
\fullfunction{\mathcal{A}}{\Omega} {X}{u}{\mathcal{A}u}
\end{equation*}
by 
\begin{equation*}
\mathcal{A}u(t) = \int\limits_0^t {{(t - s)}^{q - 1}}f(s,u(s))\,ds\,\,\, 
\text{for every}\,\,\, t\in [0,\eta].
\end{equation*}
We claim that the operator $\mathcal{T}$ is continuous and compact. It is
sufficient to show that $\mathcal{A}$ is continuous and compact.\newline

The proof consists of five steps.\newline
\textit{\textbf{\underline{Step 1.}}} For every $u\in\Omega,\,\,\,t \in
[0,a] $, we claim that $\mathcal{A}u(t) \in \mathbb{R}_F^c$. Indeed, it is
known that $\mathcal{A}u(t)\in \mathbb{R}_F$. Next, we prove $\left( {{\mathcal{A}}u(t)} \right)^\pm(\alpha)$ is continuous
 in $\alpha \in [0,1]$.\newline
Assumme that $\{\alpha_n\}$ is a sequence in $[0,1]$ converging to $\alpha$. Then, we have

\begin{itemize}
\item ${(t - s)^{q - 1}}{\left( {f(s,u(s))} \right)^ \pm }(\alpha _n) \to 
{(t - s)^{q - 1}}{\left( {f(s,u(s))} \right)^ \pm }(\alpha )$ as $n\to\infty$, for every $s\in [0,t],$

\item $\left| {{{(t - s)}^{q - 1}}{{\left( {f(s,u(s))} \right)}^ \pm }({\alpha _n})} \right| 
\le {(t - s)^{q - 1}}D(f(s,u(s)),0) \le M{(t - s)^{q - 1}}$, for every $s\in [0,t], \,\,n \in \mathbb{N}.$
\end{itemize}
By Appling Lebesgue's Dominated Convergence Theorem, we obtain that 
\begin{equation*}
\left( {{\mathcal{A}}u(t)} \right)^\pm(\alpha_n) \to \left( {{\mathcal{A}}u(t)} \right)^\pm(\alpha),\,\,\,\text{as}\,\,\, n\to \infty.
\end{equation*}
This implies that $\mathcal{A}u(t) \in \mathbb{R}_F^c$.\newline
\textit{\textbf{\underline{Step 2.}}} For fixed $u\in \Omega$, we claim that 
$\mathcal{A}u\in X.$ In fact, we will prove that $\mathcal{A}u$ is uniformly
continuous on $[0,\eta]$. Indeed, for ${t_1},{t_2} \in [0,\eta],\,\,\,{t_1}
< {t_2}$. We have 
\begin{equation}
\begin{split}
D(\mathcal{A}u({t_1}),& {}\mathcal{A}u({t_2}))\\
& {}= D\left( {\int\limits_0^{{t_2}%
} {{{({t_2} - s)}^{q - 1}}f(s,u(s)){\mkern 1mu} ds} ,\int\limits_0^{{t_1}} {{%
{({t_1} - s)}^{q - 1}}f(s,u(s)){\mkern 1mu} ds} } \right) \\
& {} \le D\left( {\int\limits_0^{{t_1}} {{{({t_2} - s)}^{q - 1}}f(s,u(s)){%
\mkern 1mu} ds} ,\int\limits_0^{{t_1}} {{{({t_1} - s)}^{q - 1}}f(s,u(s)){%
\mkern 1mu} ds} } \right) \\
& {}\,\,\,\,+ D\left( {\int\limits_{{t_1}}^{{t_2}} {{{({t_2} - s)}^{q - 1}}%
f(s,u(s))\,ds} ,0} \right) \\
& {} \le \int\limits_0^{{t_1}} {[{{({t_1} - s)}^{q - 1}} - {{({t_2} - s)}^{q
- 1}}]{\mkern 1mu} D\left( {f(s,u(s)), 0} \right)ds} \\
& {}\,\,\,\,+ \int\limits_{{t_1}}^{{t_2}} {{{({t_2} - s)}^{q - 1}}D\left( {%
f(s,u(s)),0} \right)ds} \\
& {} \le M\left\{ {\int\limits_0^{{t_1}} {[{{({t_1} - s)}^{q - 1}} - {{({t_2}
- s)}^{q - 1}}]{\mkern 1mu} ds} + \int\limits_{{t_1}}^{{t_2}} {{{({t_2} - s)}%
^{q - 1}}ds} } \right\} \\
& {}\le \frac{M}{q }[2{({t_2} - {t_1})^q} + (t_1^q - t_2^q)] \le \frac{{2M}}{%
q }{({t_2} - {t_1})^q}.
\end{split}%
\end{equation}
The above estimates shows the uniform continuity of $\mathcal{A}u$ on $%
[0,\eta].$\newline
\textit{\textbf{\underline{Step 3.}}} We claim that $\mathcal{T}u \in \Omega$%
, for every $u \in \Omega$. Indeed, for $t \in [0,\eta]$ we have 
\begin{equation}
\begin{aligned} D(\mathcal{T}u(t),0 ) & \le D({u_0}(t),0) + \frac{1}{{\Gamma
(q)}}\int\limits_0^t {{{(t - s)}^{q - 1}}D(f(s,u(s)),0) ds}\\ & {} \le N +
\frac{M}{{\Gamma (q)}}\int\limits_0^t {{{(t - s)}^{q - 1}}\,ds}\\ &{}\le N +
\frac{{M{\eta ^q}}}{{\Gamma (q + 1)}}\le R,\,\,\,\, \text{for every}\,\,\,\,
u \in \Omega. \end{aligned}
\end{equation}
This shows that $D_X (\mathcal{T}u,0 ) \le R,$ for every $u \in \Omega.$\newline
Thus, $\mathcal{T}$ maps the set $\Omega$ to itself.\newline
\textit{\textbf{\underline{Step 4.}}} Let $u,{u_n} \in \Omega :{u_n} \to u$.
We claim that $\mathcal{A}{u_n} \to \mathcal{A}u.$\newline
For $\varepsilon > 0$ arbitrarily given and small enough such that ${%
\varepsilon _0} := {\left( {\frac{{q\varepsilon }}{{2M}}} \right)^{\frac{1}{q%
}}} \in (0,\eta]$.\newline
For any $t \in (0,\eta].$ There are two cases\newline
\underline{Case 1.} $t \in (0, \varepsilon_0]$ we have 
\begin{equation}
\begin{split}
D(\mathcal{A}u_n(t),&{}\mathcal{A}u(t))\\
&{}= D\left( {\int\limits_0^t {{{(t - s)}^{q - 1}}f(s,{u_n}(s)) ds} ,\int\limits_0^t {{{(t - s)}^{q - 1}}f(s,u(s)) ds}} \right) \\
& {} \le \int\limits_0^t {{{(t - s)}^{q - 1}}D\left( {f(s,{u_n}(s)),f(s,u(s))%
} \right) ds} \\
& {} \le 2M\int\limits_0^t {{{(t - s)}^{q - 1}} ds} \\
& {} \le \frac{{2M}}{q}\varepsilon _0^q = \varepsilon.
\end{split}%
\end{equation}
\underline{Case 2.} $t \in (\varepsilon_0,\eta]$, we have 
\begin{equation}
\begin{split}
D(\mathcal{A}u_n(t), \mathcal{A}u(t))& \le \int\limits_0^t {{{(t - s)}^{q -
1}}D\left( {f(s,{u_n}(s)),f(s,u(s))} \right) ds} \\
& {} \le \int\limits_0^{t - \frac{{\varepsilon _0}}{2}} {{{(t - s)}^{q - 1}}%
D\left( {f(s,{u_n}(s)),f(s,u(s))} \right)\,ds} \\
& {}+ \int\limits_{t - \frac{{\varepsilon _0}}{2}}^t {{{(t - s)}^{q - 1}}%
D\left( {f(s,{u_n}(s)),f(s,u(s))} \right)\,ds} \\
& {} \le {\left( {\frac{{\varepsilon _0}}{2}} \right)^{q - 1}}I(\eta ) + 
\frac{{2M\varepsilon _0^q}}{{q{2^q}}},
\end{split}%
\end{equation}
where 
\begin{equation*}
I(\eta)=\int\limits_0^\eta {D\left( {f(s,{u_n}(s)),f(s,u(s))} \right) ds}.
\end{equation*}
From Lebesgue's the Dominated Convergence Theorem, we get $\mathop {\lim }%
\limits_{n \to \infty } I(\eta ) = 0$. This result shows that there is $%
n_{\varepsilon}\in \mathbb{N}$ such that $I(\eta ) \le \frac{{M\varepsilon_0 
}}{q}({2^q} - 1)$ whenever $n\geq n_{\varepsilon}$.\newline
Thus 
\begin{equation}
\begin{split}
D(\mathcal{A}u_n(t),\mathcal{A}u(t))&{}\le {\left( {\frac{{\varepsilon _0}}{2}} \right)^{q - 1}}\frac{{M{\varepsilon _0}}}{q}({2^q} - 1) + \frac{{2M\varepsilon _0^q}}{{q{2^q}}}\\
 &{}\le \frac{{2M\varepsilon _0^q}}{q} =\varepsilon,\,\,\,\text{whenever}\,\,\, n\geq n_{\varepsilon}.
\end{split}
\end{equation}
By combining (4.4)--(4.6), it is easy to deduce that 
\begin{equation*}
D_X(\mathcal{A}u_n, \mathcal{A}u)\le \varepsilon, \,\,\,\text{whenever}\,\,\, n\geq n_{\varepsilon}.
\end{equation*}
This means that $\mathcal{A}{u_n} \to \mathcal{A}u$ in $X$ as $n\to \infty$.%
\newline
\textit{\textbf{\underline{Step 5.}}} Now we claim that $\mathcal{A}(\Omega
) $ is relatively compact in $X$. Indeed, first we observe that the
equicontinuity of $\mathcal{A}(\Omega)$ is a consequence of estimate (4.2).
Therefore by using Theorem 3.1, it remains to prove that: For every $t\in
[0,\eta]$, $\mathcal{A}(\Omega)(t)$ is relatively compact in $\mathbb{R}%
_F^c$.

From the same arguments as estimate (4.3), it is easy to check that the
boundedness of $\mathcal{A}(\Omega)(t)$ holds for every $t\in [0,\eta]$. By
using Theorem 3.3 we only need to prove that $\mathcal{A}(\Omega)(t)$ is
level - equicontinuous, for every $t\in [0,\eta]$.

It follows from the hypothesis $(H_2)$ that $f((0,\eta ] \times B_R)$ is
relatively compact in $\mathbb{R}_F^c$. Also by using Theorem 3.3 we deduce $%
f((0,\eta ] \times B_R)$ is level-equicontinuous. Therefore for every $%
\varepsilon > 0$ there exists $\delta > 0$ such that for all $%
\alpha,\,\,\,\beta \in [0,1]$. If  $|\alpha - \beta | < \delta$ then
\begin{equation*}
 D\left( {{{\left[ {f(s,u(s))} \right]}%
^\alpha },{{\left[ {f(s,u(s))} \right]}^\beta }} \right) < \frac{{q
\varepsilon }}{{\eta^q}},\;\;\,{\text{for every}}\,\,(s,u) \in (0,\eta ]
\times \Omega.
\end{equation*}
Hence, for every $u\in \Omega$, $t\in [0,\eta]$ we get 
\begin{equation*}
\begin{split}
D({[{\mathcal{A}}u(t)]^\alpha },{[{\mathcal{A}}u(t)]^\beta })&{}\le
\int\limits_0^t {{{(t - s)}^{q - 1}}} D\left( {{{[f(s,u(s))]}^\alpha },{{%
[f(s,u(s))]}^\beta }} \right)ds \\
& {}< \frac{{q\varepsilon }}{{\eta ^q}}\int\limits_0^t {{{(t - s)}^{q - 1}}} 
{\mkern 1mu} ds \le \varepsilon,{\text{whenever}} \;\;\,|\alpha - \beta | <
\delta.
\end{split}%
\end{equation*}
This implies that $\mathcal{A}(\Omega)(t)$ is level-equicontinuous for every 
$t\in [0,\eta].$\newline

So $\mathcal{A}(\Omega)$ is relatively compact in $X$, i.e, $\mathcal{A}$ is a
compact operator, we deduce that $\mathcal{T}$ is too.\newline

The results obtained from the above steps allow us to conclude that $%
\mathcal{T}$ is continuous and compact. By Theorem 2.5 $\mathcal{T}$ has a
fixed point in $\Omega$ i.e, the integral equation (4.1) has a solution.
\end{proof}

\section{\textbf{Generalized problem}}

The solvability of the problem (4.1) still holds if the space $\mathbb{R}%
_{F}^{c}$ is replaced by $\mathbb{E}_{c}^{n}$. Indeed the whole proof is
completely similar by using a criterion of compactness for the space $
\mathbb{E}_{c}^{n}$ which we propose as follows. 
\begin{theorem}
A subset $\mathcal{B}$ of $\mathbb{E}_c^n$ is  relatively compact  \textbf{iff}
\begin{enumerate}
\item [i)] $\mathcal{B}$ is bounded in $\mathbb{E}_c^n$,
\item [ii)] $\mathcal{B}$ is level - equicontinuous.
\end{enumerate}
\end{theorem}

\begin{proof}
First we need the following lemmas 
\begin{lemma} (see \cite{LG})
A subset $\Gamma$ of $K_c(\mathbb{R}^n)$ is relatively compact \textbf{iff} $\Gamma$ is bounded.
\end{lemma}%
\begin{lemma}
Let $\mathcal{B}$ be a subset of $\mathbb{E}_c^n$. Put  $\mathcal{B}(\alpha)=\{[u]^\alpha | u\in \mathcal{B} \}$ for every
$\alpha\in [0,1]$. Then the following satatements are equivalent
\begin{enumerate}
\item [i)] for every $\alpha\in [0,1], \mathcal{B}(\alpha)$ is relatively compact in $K_c(\mathbb{R}^n)$, 
\item [ii)]$\mathcal{B}$ is bounded in $\mathbb{E}_c^n.$
\end{enumerate}
\end{lemma}

\begin{proof} Proof of Lemma 5.3\newline
i) $\Rightarrow $ ii) By the Lemma 5.2, $\mathcal{B}(0)$ is bounded in $K_{c}(%
\mathbb{R}^{n})$, this means that there is $M>0$ such that 
\begin{equation*}
\begin{split}
& {}\,\,\,\,\,\,\,\,\,\, d_{H}([u]^{0},0)\leq M,\,\,\forall u\in \mathcal{B} \\
& {}\Rightarrow d_{H}([u]^{\alpha },0)\leq M,\,\,\forall \alpha \in \lbrack
0,1],\,\,\forall u\in \mathcal{B} \\
& {}\Rightarrow D(u,0)\leq M,\,\,\forall u\in \mathcal{B}.
\end{split}%
\end{equation*}%
This implies that $\mathcal{B}$ is bounded in $\mathbb{E}_{c}^{n}.$\newline

ii) $\Rightarrow $ i) It follows from the boundedness of $\mathcal{B}$ that
there is $M>0$ such that%
\begin{equation*}
\begin{split}
& {}\,\,\,\,\,\,\,\,\,\, D(u,0)\leq M,\,\,\forall u\in \mathcal{B}\\
& {}\Rightarrow d_{H}([u]^{\alpha },0)\leq M,\,\,\forall \alpha \in \lbrack
0,1],\,\,\forall u\in \mathcal{B}\\
& {}\Rightarrow d_{H}(x,0)\leq M,\,\,\,\forall x\in \mathcal{B}(\alpha),\,\,\,\forall \alpha \in \lbrack
0,1].
\end{split}%
\end{equation*}%
This implies that $\mathcal{B}(\alpha )$ is bounded in $K_{c}(\mathbb{R}^{n})$,
for every $\alpha \in \lbrack 0,1].$

Using Lemma 5.2 we end the proof of Lemma 5.3.
\end{proof}

Now we are in position to prove Theorem 5.1. It is easy to see that $\mathbb{%
E}_{c}^{n}$ is a closed subspace of the metric space $C([0,1],K_{c}(\mathbb{R%
}^{n}))$ with respect to the metric 
\begin{equation*}
D(u,v)=\mathop {\sup }\limits_{\alpha \in \lbrack 0,1]}{d_{H}}({u(\alpha )},{%
v(\alpha )}).
\end{equation*}%
By using Theorem 3.1 we obtain the following result:\newline
$\mathcal{B} \subset \mathbb{E}_{c}^{n}$ is relatively compact in $C([0,1],K_{c}(%
\mathbb{R}^{n}))$ (i.e. in $\mathbb{E}_{c}^{n}$ because of the closedness of $\mathbb{E}_c^n$)
\textbf{iff}
\begin{enumerate}
\item[a)] for every $\alpha\in [0,1],\,\,\mathcal{B}(\alpha)$ is relatively compact in $K_c(\mathbb{R}^n)$,

\item[b)] $\mathcal{B}$ is equicontinuous (just be level - equicontinuous).
\end{enumerate}

Using the Lemma 5.3 we see that the part a) is equivalent to the part i) of
Theorem 5.1 and we end the proof of  Theorem 5.1.
\end{proof}

\begin{remark}\rm
By the same technique solving fuzzy integral equation (4.1), we can deal with a number of fuzzy 
fractional differential equations. 
\end{remark}

\end{document}